\documentclass[letterpaper]{amsart} 

\usepackage{amsmath}
\usepackage[all]{xy}

\usepackage{psfrag}
\usepackage{epsfig}
 
\numberwithin{equation}{section}

\newtheorem{theorem}{Theorem}[section] 
\newtheorem{proposition}[theorem]{Proposition} 
 
\newtheorem{corollary}[theorem]{Corollary} 
\newtheorem{lemma}[theorem]{Lemma} 
 
\theoremstyle{definition}

\newtheorem{definition}[theorem]{Definition}

\newcommand{\credit}[1]{\smallskip\noindent {\textbf{#1.\ }}}

\def\cc{\mathbf{c}}

\def\yy{\mathbf{y}}
\def\TT{\mathbb{T}}

\def\ZZ{\mathbb{Z}}

\newcommand{\overunder}[2]{
\!\begin{array}{c}
\scriptstyle{#1}\\[-.1in]
-\!\!\!-\!\!\!-\\[-.1in]
\scriptstyle{#2}
\end{array}
\!
}
\begin{document}

\title{Cluster algebras and symmetric matrices}

\author{Ahmet I. Seven}

\address{Middle East Technical University, Department of Mathematics, 06800, Ankara, Turkey}
\email{aseven@metu.edu.tr}

\thanks{The author's research was supported in part by the Turkish Research Council (TUBITAK)}


\date{April 10, 2013}



\begin{abstract}
In the structural theory of cluster algebras, a crucial role is played by a family of integer vectors, called {$\cc$-vectors}, which parametrize the coefficients. It has recently been shown that each $\cc$-vector with respect to an acyclic initial seed is a real root of the corresponding root system. In this paper, we obtain an interpretation of this result in terms of symmetric matrices. We show that, for skew-symmetric cluster algebras, the $\cc$-vectors associated with any seed defines a quasi-cartan companion for the corresponding exchange matrix, i.e. they form a companion basis, and we establish some basic combinatorial properties. In particular, we show that these vectors define an admissible cut of edges in the associated quivers.


\end{abstract}
\maketitle

\section{Introduction}
\label{sec:intro}

In the theory of cluster algebras, a central role is played by a family of elements called 
\emph{coefficients}. These elements are parametrized by a family of integer vectors called \emph{$\cc$-vectors} \cite{CAIV,NZ}.  
It has recently been shown in \cite{ST} that each $\cc$-vector with respect to an initial acyclic seed is a real root of the corresponding root system. In this paper, we obtain an interpretation of this result in terms of symmetric matrices.

To be more specific, let us recall some terminology. In the theory of cluster algebras, there is a notion of a $Y$-seed, which is defined as a tuple $(\yy,B)$ where $\yy$ is a coefficient tuple and $B$ is a skew-symmetrizable matrix (i.e. an $n\times n$ inetger matrix such that $DB$ is skew-symmetric for some diagonal matrix $D$ with positive diagonal entries). On the other hand, the coefficient tuple $\yy$ is determined by the corresponding $\cc$-vectors, which are conjectured to have a sign coherence property \cite{CAIV}. Since, in this paper, we deal with the combinatorial properties of $\cc$-vectors and we do not need algebraic properties of cluster algebras, it is convenient for us to slightly abuse the terminology and call a \emph{$Y$-seed}
a tuple $(\cc, B)$, where $B$ is a skew-symmetrizable integer matrix and $\cc=(\cc_1,...,\cc_n)$, where each $\cc_i=(c_1,...,c_n) \in \ZZ^n$ is non-zero and has the following sign coherence property: the vector $\cc_{i}$ has either all entries nonnegative or all entries nonpositive; we write $sgn(\cc_i)=+1$ or $sgn(\cc_i)=-1$ respectively. We refer to $B$ as the \emph{exchange matrix} of a $Y$-seed and $\cc$ as the \emph{$\cc$-vector} tuple. 
We also use the notation $[b]_+ = \max(b,0)$. For $k = 1, \dots, n$, the \emph{$Y$-seed mutation} $\mu_k$ transforms
$(\cc, B)$ into the tuple $\mu_k(\cc, B)=(\cc', B')$ defined as follows \cite[Equation~(5.9)]{CAIV}:
\begin{itemize}
\item
The entries of the exchange matrix $B'=(B'_{ij})$ are given by
\begin{equation}
\label{eq:matrix-mutation}
B'_{ij} =
\begin{cases}
-B_{ij} & \text{if $i=k$ or $j=k$;} \\[.05in]
B_{ij} + [B_{ik}]_+ [B_{kj}]_+ - [-B_{ik}]_+ [-B_{kj}]_+
 & \text{otherwise.}
\end{cases}
\end{equation}
\item
The tuple $\cc'=(\cc_1',\dots,\cc_n')$ is given by
\begin{equation}
\label{eq:y-mutation}
\cc'_i =
\begin{cases}
-\cc_{i} & \text{if $i = k$};\\[.05in]
\cc_i+[sgn(\cc_k)B_{k,i}]_+\cc_k
 & \text{if $i \neq k$}.
\end{cases}
\end{equation}
\end{itemize}

It is easy to see that $B'$ is skew-symmetrizable (with the same choice of $D$). It is not clear, however, that $\cc'=(\cc_1',\dots,\cc_n')$ has the sign coherence property, so $\mu_k(\cc, B)$ may not be a $Y$-seed. However, this will not be a problem for us because, in this paper, we will study $Y$-seeds $(\cc, B)$ for which $\mu_k(\cc, B)$ is a $Y$-seed; 
note then that $\mu_k$ is involutive (that is, it transforms $(\cc', B')$ into the original $Y$-seed $(\cc, B)$). We shall also use the notation $B' = \mu_k(B)$ (resp.~$(\cc', B') = \mu_k(\cc, B)$) and call the transformation $B \mapsto B'$ the \emph{matrix mutation}. 
This operation is involutive, so it defines a \emph{mutation-equivalence} relation on skew-symmetrizable matrices.

We use $Y$-seeds in an association with vertices of an $n$-regular tree. To be more precise, let $\TT_n$ be an \emph{$n$-regular tree} whose edges are labeled by the numbers $1, \dots, n$, so that the $n$ edges emanating from each vertex receive different labels. We write $t \overunder{k}{} t'$ to indicate that vertices $t,t'\in\TT_n$ are joined by an edge labeled by~$k$.
A \emph{$Y$-seed pattern} is an assignment
of a seed $(\cc_t, B_t)$
to every vertex $t \in \TT_n$, such that the seeds assigned to the
endpoints of any edge $t \overunder{k}{} t'$ are obtained from each
other by the seed mutation~$\mu_k$.
Following \cite{NZ}, we write:
\begin{equation}
\label{eq:seed-labeling}
\cc_t = (\cc_{1;t}\,,\dots,\cc_{n;t})\,,\quad
B_t = (B_{ij;t})\,.
\end{equation}

Let us note that a seed pattern, if exists, is uniquely determined by a fixed initial seed at a vertex $t_0$ in $\TT^n$. However, the existence of a $Y$-seed pattern is far from being trivial. It is conjectured that the particular choice of the initial $Y$-seed $(\cc_0,B_0)$, where $\cc_0$ is the tuple of standard basis vectors determines a $Y$-seed pattern \cite{CAIV}. 
This conjecture has been proved for skew-symmetric $B_0$ \cite{DWZ2}. This will be sufficient for us in this paper as we will not consider skew-symmetrizable matrices; we will be working only with skew-symmetric matrices and the associated $Y$-seed patterns. In an important special case  the $\cc$-vectors have a particular property \cite{RS}, which we will recall after a bit more preparation. Let us recall that
the \emph{diagram} of a skew-symmetrizable $n\times n$ matrix ${B}$ is the directed graph $\Gamma ({B})$ defined as follows: the vertices of $\Gamma ({B})$ are the indices $1,2,...,n$ such that there is a directed edge from $i$ to $j$ if and only if ${B}_{j,i} > 0$, and this edge is assigned the weight $|B_{ij}B_{ji}|\,$. By a {subdiagram} of $\Gamma(B)$, we always mean a diagram obtained from $\Gamma(B)$ by taking an induced (full) directed subgraph on a subset of vertices and keeping all its edge weights the same as in $\Gamma(B)$. By a cycle in $\Gamma(B)$ we mean a subdiagram whose vertices can be labeled by elements of $\ZZ/m\ZZ$ so that the edges betweeen them are precisely $\{i,i+1\}$ for $i \in  \ZZ/m\ZZ$. Let us also note that if $B$ is skew-symmetric then it is also represented, alternatively, by a quiver whose vertices are the indices $1,2,...,n$ and there are $B_{j,i}>0$ many arrows from $i$ to $j$. This quiver uniquely determines the corresponding skew-symmetric matrix, so mutation of skew-symmetric matrices can be viewed as a "quiver mutation".

We call a $Y$-seed $(\cc,B)$ \emph{acyclic} if $\Gamma(B)$ is acyclic, i.e. has no oriented cycles. Let us also recall that, for a skew-symmetrizable $B$ with an acylic diagram $\Gamma(B)$, there is a corresponding generalized Cartan matrix $A$ such that $A_{i,i}=2$ and $A_{i,j}=-|B_{i,j}|$ for $i\ne j$. Then there is an associated root system in the root lattice spanned by the simple roots $\alpha_i$ \cite{K}. For each simple root $\alpha_i$, the corresponding reflection $s_{\alpha_i}=s_i$ is the linear isomorphism defined on the basis of simple roots as $s_i(\alpha_j)=\alpha_j-A_{i,j}\alpha_i$. Then the real roots are defined as the vectors obtained from the simple roots by a sequence of reflections. It is well known that the coordinates of a real root with respect to the basis of simple roots are either all nonnegative or all nonpositive, see \cite{K} for details. We can now recall the main statements that we use in this paper. 

\begin{theorem}\label{th:DWZ2-RS} \cite[Theorem~1.4]{ST}, \cite[Theorem~1.7]{DWZ2} Suppose that the initial seed $\Sigma_0=(\cc_0,B_0)$ is acyclic with $\cc_{0}$ being the standard basis. Then $\Sigma_0$ uniquely determines a $Y$-seed pattern $\Sigma_t$, $t \in \TT_n$; furthermore each $\cc$-vector $\cc_{i;t}$ is the coordinate vector of a real root with respect to the basis of simple roots in the corresponding root system.
\end{theorem}

In this paper, we obtain an interpretation of this result in terms of symmetric matrices, using the notion of a quasi-Cartan companion of a skew-symmetrizable matrix \cite{BGZ}. 
Let us recall that an $n\times n$  matrix $A$ is called symmetrizable if there exists a symmetrizing diagonal matrix $D$ with positive diagonal entries such that $DA$ is symmetric. A crucial property of $A$ is sign symmetry: $sgn(A_{i,j})=sgn(A_{j,i})$. We say that $A$ is a \emph{quasi-Cartan matrix} if it is symmetrizable and all of its diagonal entries are equal to $2$. 
A \emph{quasi-Cartan companion} (or "companion" for short) of a skew-symmetrizable matrix $B$ is a quasi-Cartan matrix $A$ with $|A_{i,j}|= |B_{i,j}|$ for all $i \ne j$. The basic example of a quasi-Cartan companion of $B$ is the associated generalized Cartan matrix $A$, which is defined as $A_{i,j}= -|B_{i,j}|$, for all $i\ne j$. In this paper, we use a variation of this construction by choosing the signs of the entries in relation with the structure of the associated diagram. More precisely, we call a quasi-Cartan companion $A$ of a skew-symmetrizable matrix $B$ \emph{admissible} if, for any oriented (resp. non-oriented) cycle $Z$ in $\Gamma(B)$, there is exactly an odd (resp. even) number of edges $\{i,j\}$ such that $A_{i,j}>0$. An arbitrary skew-symmetrizable matrix need not have an admissible quasi-Cartan companion; however, if $\Gamma(B)$ is acyclic, then it has an admissible quasi-Cartan companion, the associated  generalized Cartan matrix. Our main result is a generalization of this fundamental property: 

\begin{theorem}\label{th:admissible}
Suppose that $B$ is a skew-symmetric matrix which is mutation-equivalent to $B_0$ such that $\Gamma(B_0)$ is acyclic. Then $B$ has an admissible quasi-Cartan companion.
\end{theorem}

We obtain this statement by establishing an admissible quasi-Cartan companion using $\cc$-vectors as follows:





\begin{theorem}\label{th:companion0}
Let $\Sigma_t=(\cc,B)$ be a $Y$-seed with respect to an acyclic initial seed $\Sigma_0=(\cc_0,B_0)$ such that $B_0$ is skew-symmetric (and $\cc_{0}$ is the standard basis). Let $A_0$ be the  (symmetric) generalized Cartan matrix associated to $B_0$. 
Then $A_t=A=(\cc_i^TA_0\cc_j)$ is a quasi-Cartan companion of $B$ \footnote{In the terminology of \cite{P,BM}, the family $\cc_t$ gives rise to a ''companion basis'' associated to $B_t$}. Furthermore, the matrix $A$ has the following properties:
\begin{itemize}
\item
If $sgn(B_{j,i})=sgn(\cc_j)$, then $\cc_j^TA_0\cc_i=-sgn(\cc_j)B_{j,i}=-|B_{j,i}| $.
\item
If $sgn(B_{j,i})=-sgn(\cc_j)$, then $\cc_j^TA_0\cc_i=sgn(\cc_i)B_{j,i}=-sgn(\cc_i)sgn(\cc_j)|B_{j,i}|$.
\end{itemize}
(here $\cc_i^T$ denotes the transpose of $\cc_i$ viewed as a column vector.)

\smallskip
In particular; if $sgn(\cc_j)=-sgn(\cc_i)$, then $B_{j,i}=sgn(\cc_i)\cc_j^TA_0\cc_i$.





\end{theorem}




We also obtain some basic combinatorial properties, including admissibility, of these quasi-Cartan companions defined by $\cc$-vectors:
\begin{theorem}\label{th:companion2}
The quasi-Cartan companion $A$ from Theorem~\ref{th:companion0} has the following properties:

\begin{itemize}
\item
Every directed path of the diagram $\Gamma(B)$ has at most one edge $\{i,j\}$ such that $A_{i,j}>0$.
\item
Every oriented cycle of the diagram $\Gamma(B)$ has exactly one edge $\{i,j\}$ such that $A_{i,j}>0$.

\item
Every non-oriented cycle of the diagram $\Gamma(B)$ has an even number of edges $\{i,j\}$ such that $A_{i,j}>0$.
\end{itemize}

In particular, the quasi-Cartan companion $A$ is admissible. Furthermore, any admissible quasi-Cartan companion of $B$ can be obtained from $A$ by a sequence of simultaneous sign changes in rows and columns. 

\end{theorem}


We can describe this result in terms related to the theory of quivers with potential. More precisely,
motivated by \cite{HI}, let us call a set $C$ of edges in $\Gamma_B$ an ''admissible cut'' if every oriented cycle contains exactly one edge that belongs to $C$ and every non-oriented cycle contains exactly an even number of edges in $C$. (Note that the definition of a cut of edges in \cite{HI} does not have a condition on non-oriented cycles). Then, for $B$ as in Theorem~\ref{th:companion0}, its diagram (or quiver) $\Gamma(B)$ has an admissible cut of edges as follows:
 
\begin{corollary}\label{cor:cut}
In the set-up of Theorem~\ref{th:companion0}, let $C$ be the edges $\{i,j\}$ in $\Gamma(B)$ such that $A_{i,j}>0$. Then $C$ is an admissible cut. 
\end{corollary}
\noindent
Let us also note that, in the theory of cluster categories, quivers that can be obtained from an acyclic quiver by a sequence of mutations is called ''cluster tilted'' \cite{BMR}. Thus, every cluster tilted quiver has an admissible cut of edges.

We also give an interpretation of the $\cc$-vectors in terms related to quasi-Cartan companions. For this purpose, let us first discuss an extension of the mutation operation to quasi-Cartan companions:
\begin{definition} 
\label{def:comp-mut} 
Suppose that $B$ is a skew-symmetrizable matrix and let $A$ be a quasi-Cartan companion of $B$. 
Let $k$ be an index. For each sign $\epsilon=\pm1$, "the $\epsilon$-mutation of $A$ at $k$" is the quasi-Cartan matrix 
$\mu^\epsilon(A)=A'$ such that for any $i,j \ne k$: $A'_{i,k}=\epsilon sgn(B_{k,i})A_{i,k}$, $A'_{k,j}=\epsilon sgn(B_{k,j})A_{k,j}$, $A'_{i,j}=A_{i,j}-sgn(A_{i,k}A_{k,j})[B_{i,k}B_{k,j}]_+$. \footnote{Note that $A'$ may not be a quasi-Cartan companion of $B'=\mu_k(B)$, see Proposition~\ref{prop:adm-triangle}.}
\end{definition}

\noindent
Note that for $\epsilon=-1$, one obtains the formula in \cite[Proposition~3.2]{BGZ}. Also note that if $D$ is a skew-symmetrizing matrix of $B$, then $D$ is also a symmetrizing matrix for $A$, with $DA=S$ symmetric. If we consider $S$ as the Gram matrix of a symmetric bilinear form on $\ZZ^n$ with respect to a basis $\mathcal{B}=\{e_1,...,e_n\}$, then $DA'=S'$ is the Gram matrix of the same symmetric bilinear form with respect to the basis $\mathcal{B'}=\{e'_1,e'_2,...,e'_n\}$ defined as follows: $e'_k=-e_k$; $e'_i=e_i-A_{k,i}e_k$ if $\epsilon B_{k,i}>0$; $e'_i=e_i$ if else. We write $\mu_k(\mathcal{B})=\mathcal{B}'$. We discuss some other properties of the mutations of quasi-Cartan companions in Section~\ref{sec:pre}.

\begin{corollary}\label{cor:companion}
Let $\Sigma_t=(\cc,B)$ be a $Y$-seed with respect to an acyclic initial seed $\Sigma_0=(\cc_0,B_0)$ such that $B_0$ is skew-symmetric. Let $\epsilon=sgn(\cc_k)$. Then, for $t \overunder{k}{} t'$ with $\Sigma_{t'}=\mu_k(\cc, B)=(\cc', B')$, we have the following:
\begin{itemize}
\item
$\cc'=\mu^{\epsilon}_k(\cc)$,  
\item
$A_{t'}=\mu^{\epsilon}_k(A_t)$.
\item
if $\cc'_i\ne \cc_i$, then $\cc'_i=s_{\cc_k}(\cc_i)$, where $s_{\cc_k}$ is the reflection with respect to the real root $\cc_k$ and $\ZZ^n$ is identified with the root lattice.
\end{itemize}



\end{corollary}

For an arbitrary admissible quasi-Cartan companion, we have the following property: 

\begin{corollary}\label{cor:admissible2}
Suppose that $B$ is a skew-symmetric matrix which is mutation-equivalent to $B_0$ such that $\Gamma(B_0)$ is acyclic. Suppose also that $S$ is an admissible quasi-Cartan companion of $B$. Then, for any $\epsilon=\pm 1$, the matrix $S'=\mu^\epsilon_k(S)$ is an admissible quasi-Cartan companion of $B'=\mu_k(B)$.  

\end{corollary}

As an application of our results, let us note that Theorem~\ref{th:admissible} could be useful for recognizing quivers which can be obtained from an acyclic one by a sequence of mutations (i.e. cluster tilted quivers): if a quiver (viewed as the diagram of a skew-symmetric matrix) does not have an admissible quasi-Cartan companion, then it can not be obtained from any acyclic quiver by a sequence of mutations.
For example, the quiver in Figure~\ref{fig:nonadm} can not be obtained from any acyclic quiver by a sequence of mutations because it does not have an admissible quasi-Cartan companion.

In accordance with the general conjectures on $\cc$-vectors \cite{CAIV,RS}, we conjecture that the above results also hold for skew-symmetrizable matrices.
We prove our results in Section~\ref{sec:proof} after some preparation in Section~\ref{sec:pre}.

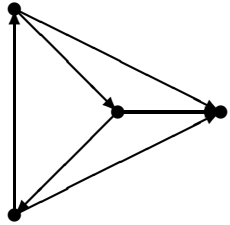
\begin{figure}[ht]

\setlength{\unitlength}{2.6pt}

\begin{center}

\begin{picture}(60,30)(-30,0)



\thicklines

\put(-30,0){\circle*{2.0}}
\put(-15,15){\circle*{2.0}}
\put(-30,30){\circle*{2.0}}
\put(0,15){\circle*{2.0}}




\put(-30,0){\vector(0,1){30}}
\put(-30,30){\vector(2,-1){30}}
\put(-30,30){\vector(1,-1){15}}
\put(-30,0){\vector(2,1){30}}
\put(-15,15){\vector(-1,-1){15}}
\put(-15,15){\vector(1,0){15}}

\end{picture}

\end{center}

\caption{a diagram which does not have an admissible quasi-Cartan companion} 

\label{fig:nonadm}

\end{figure}

\section{Preliminaries}
\label{sec:pre}

In this section, we will recall and prove some statements that we will use to prove our results. First, let us note the following properties of the mutation of a quasi-Cartan companions, which can be easily checked using the definitions: 

\begin{proposition} 
\label{prop:+-} 
Suppose that $B$ is a skew-symmetrizable matrix and let $A$ be a quasi-Cartan companion of $B$. Then we have the following:

(i) $\mu_k^\epsilon(A)$ and $\mu_k^{-\epsilon}(A)$ can be obtained from each other by simultaneously multiplying the $k$-th row and the $k$-th column by $-1$.

(ii) Suppose that $A'=\mu_k^\epsilon(A)$ is a quasi-Cartan companion of $\mu_k(B)$. Then $\mu_k^{-\epsilon}(A')=A$. 


\end{proposition} 

Note that for $\epsilon=-1$, the mutation $\mu_k^\epsilon(A)$ has been defined in \cite[Proposition~3.2]{BGZ}, where the following statement has also been given:
\begin{proposition} 
\label{prop:adm-triangle} 
Suppose that $B$ is a skew-symmetrizable matrix and let $A$ be a quasi-Cartan companion of $B$. Then $A'=\mu_k^\epsilon(A)$ is a quasi-Cartan companion of $\mu_k(B)$ if and only if, for any triangle $T$ in $\Gamma(B)$ that contains $k$, the following holds: if $T$ is oriented (resp. non-oriented), there is exactly an odd (resp. even) number of edges $\{i,j\}$ such that $A_{i,j}>0$.

In particular, if $A$ is admissible, then $A'$ is a quasi-Cartan companion of $\mu_k(B)$.
\end{proposition} 

For convenience, we prove the following statement, which is a part of Theorem~\ref{th:companion0}:
\begin{proposition} 
\label{prop:prop1} 
Let $\Sigma_t=(\cc,B)$ be a $Y$-seed with respect to an acyclic initial seed $\Sigma_0=(\cc_0,B_0)$ such that $B_0$ is skew-symmetric. Let $A_0$ be the  (symmetric) generalized Cartan matrix associated to $B_0$. 
Then we have the following:

\begin{itemize}
\item
If $sgn(B_{j,i})=sgn(\cc_j)$, then $\cc_j^TA_0\cc_i=-sgn(\cc_j)B_{j,i}=-|B_{j,i}| $.
\item
If $sgn(B_{j,i})=-sgn(\cc_j)$, then $\cc_j^TA_0\cc_i=sgn(\cc_i)B_{j,i}=-sgn(\cc_i)sgn(\cc_j)|B_{j,i}|$.

\end{itemize}
\noindent
(In particular; if $B_{j,i}\ne 0$, then $\cc_j^TA_0\cc_i=\mp|B_{j,i}|$.)
\end{proposition} 

\noindent
(Note that the proposition does not say anything if $B_{j,i}=0$.)

\smallskip
\begin{proof} 
To prove the first part, let us suppose that $sgn(B_{j,i})=sgn(\cc_j)$. 
Let $\mu_j(\cc,B)=(\cc',B')$ with $B'=\mu_j(B)$. Then $\cc_i'=\cc_i+[sgn(\cc_j)B_{j,i}]_+\cc_j=\cc_i+sgn(B_{j,i})B_{j,i}\cc_j=\cc_i+|B_{j,i}|\cc_j$. 
Note that, since $\cc_i$, $\cc_j$ and $\cc'_i$ are coordinate vectors of real roots, we have 
$2={\cc'_{i}}^TA_0{\cc'}_i=\cc_{i}^TA_0\cc_i=\cc_{j}^TA_0\cc_j$. Also $\cc_j^TA_0\cc_i=\cc_i^TA_0\cc_j$ because $A$ is symmetric.
Then
$2={\cc'_{i}}^TA_0\cc'_i=(\cc_i+|B_{j,i}|\cc_j)^TA_0(\cc_i+|B_{j,i}|\cc_j)=
\cc_i^TA_0\cc_i+\cc_i^TA_0|B_{j,i}|\cc_j+|B_{j,i}|\cc_j^TA_0\cc_i+\\+|B_{j,i}|\cc_j^TA_0|B_{j,i}|\cc_j
=\cc_i^TA_0\cc_i+2|B_{j,i}|\cc_j^TA_0\cc_i+|{B}_{j,i}|^2\cc_j^TA_0\cc_j
=2+2|B_{j,i}|\cc_j^TA_0\cc_i+|B_{j,i}|^22=\\=2+2|B_{j,i}|(\cc_j^TA_0\cc_i+|B_{j,i}|)$, 
implying that $\cc_j^TA_0\cc_i+|B_{j,i}|=0$ thus $\cc_j^TA_0\cc_i=-|B_{j,i}|=sgn(B_{j,i})B_{j,i}=sgn(\cc_j)B_{j,i}$. 

\smallskip

To prove the second part, let us suppose that $sgn(B_{j,i})=-sgn(\cc_j)$. 
Let $\mu_i(\cc,B)=(\cc',B')$ with $B'=\mu_i(B)$. Note that $sgn(B'_{j,i})=-sgn(B_{j,i})$ (by the definition of mutation). 

First assume that $sgn(\cc_j)=-sgn(\cc_i)$. Then ${\cc'}_j=\cc_j$ and ${\cc'}_i=-\cc_i$, so $sgn({\cc'}_j)=sgn({\cc}_j)=-sgn(B_{j,i})=sgn(B'_{j,i})$, i.e. for the $Y$-seed $(\cc',B')$, we have $sgn(B'_{j,i})=sgn({\cc'}_j)$.
Thus, by the first part, we have $-|B'_{j,i}|={\cc'}_j^TA_0{\cc'}_i=-{\cc}_j^TA_0{\cc}_i$. Thus $\cc_j^TA_0\cc_i=|B'_{j,i}|=|B_{j,i}|=-sgn(\cc_i)sgn(\cc_j)|B_{j,i}|$. 

Similarly, if $sgn(\cc_j)=sgn(\cc_i)$, then ${\cc'}_j=\cc_j+|B_{i,j}|\cc_i$, and ${\cc'}_i=-\cc_i$. Thus 
$sgn({\cc'}_j)=sgn({\cc}_j)=-sgn(B_{j,i})=sgn(B'_{j,i})$, i.e. for the $Y$-seed $(\cc',B')$ we have $sgn(B'_{j,i})=sgn({\cc'}_j)$. Then, by the first part, we have $-|B'_{j,i}|={\cc'}_j^TA_0{\cc'}_i=(\cc_j+|B_{i,j}|\cc_i)^TA_0(-{\cc}_i)=-{\cc_j}^TA_0{\cc}_i-|B_{i,j}|{\cc_i}^TA_0{\cc}_i=
-{\cc_j}^TA_0{\cc}_i-2|B_{i,j}|$. Thus, since $B'_{j,i}=-B_{j,i}$, we have $\cc_j^TA_0\cc_i=-|B'_{j,i}|=-|B_{j,i}|=-sgn(\cc_i)sgn(\cc_j)|B_{j,i}|$. 

On the other hand, our assumption $sgn(B_{j,i})=-sgn(\cc_j)$ implies the following: $-sgn(\cc_i)sgn(\cc_j)|B_{j,i}|=-sgn(\cc_i)sgn(\cc_j)sgn(B_{j,i})B_{j,i}=-sgn(\cc_i)sgn(\cc_j)(-sgn(\cc_j))B_{j,i}=sgn(\cc_i)B_{j,i}$. This completes the proof.





 

\end{proof}

\section{Proofs of main results}
\label{sec:proof}
To prove our results, we first prove some lemmas for convenience, establishing some necessary conditions of the main results.  
\begin{lemma} 
\label{lem:lem-A'c'} 
In the set-up of Theorem~\ref{th:companion0}, suppose that $A=(\cc_i^TA_0\cc_j)$ is a quasi-Cartan companion of $B$. Let $B'=\mu_k(B)$ and let $A'=\mu_{k}^{\epsilon}(A)$ where $\epsilon=sgn(\cc_k)$. Then $A'_{i,j}=({\cc'}_i^TA_0{\cc'}_j)$.

Furthermore, if $\cc'_i\ne \cc_i$, then $\cc'_i=s_{\cc_k}(\cc_i)$, where $s_{\cc_k}$ is the reflection with respect to the real root $\cc_k$ and $\ZZ^n$ is identified with the root lattice.
\end{lemma} 

\begin{proof} 
Let us note that for $\mu_k(\cc,B)=(\cc',B')$ we have the following: $\cc'_k=-\cc_k$; $\cc'_i =\cc_i+[sgn(\cc_k)B_{k,i}]_+\cc_k$ if $i \neq k$ by \eqref{eq:y-mutation}. 
On the other hand, $[sgn(\cc_k)B_{k,i}]_+\ne 0$ if and only if $sgn(\cc_k)B_{k,i}>0$ if and only if $sgn(\cc_k)=sgn(B_{k,i})$. Then, by Proposition~\ref{prop:prop1}, $[sgn(\cc_k)B_{k,i}]_+=-A_{k,i}$. Thus, by Definition~\ref{def:comp-mut}, we have $\cc'_k=\mu_k(\cc_k)$ and  
$A'_{i,j}=({\cc'}_i^TA_0{\cc'}_j)$. Furthermore, (for $i$ with $[sgn(\cc_k)B_{k,i}]_+\ne 0$), we have $\cc'_i=s_{\cc_k}(\cc_i)=\cc_i-A_{k,i}\cc_k$ by the well-known properties of real roots, see \cite[Chapter~5]{K}.
\end{proof} 

\begin{lemma} 
\label{lem:lem-cycle0} 
In the set-up of Theorem~\ref{th:companion0}, suppose that $A=(\cc_i^TA_0\cc_j)$ is a quasi-Cartan companion of $B$. Then any directed path or oriented cycle contains at most one edge $\{i,j\}$ such that $A_{i,j}>0$.
\end{lemma}

\begin{proof} 
Let us first assume without loss of generality that $P: 1\to 2\to...\to r$, $r\geq 3$, 
is a directed path in $\Gamma$ such that the entries $A_{1,2},A_{{r-1},r}>0$. 
Let us first assume that $r=3$. Then, since $P$ is equioriented, we have $sgn(B_{2,1})=-sgn(B_{{2},3})$, so either 
$sgn(B_{2,1})=sgn(\cc_2)$ or $sgn(B_{2,3})=sgn(\cc_2)$. Then, by the first part of Proposition~\ref{prop:prop1}, we have $A_{2,1}<0$ or $A_{2,3}<0$, which contradicts our assumption.  
Let us now assume that $r>3$. Then, for $B'=\mu_{2}(B)$, the subdiagram $P'=P-\{2\}=1\to 3\to 4..\to r$ is a directed path in $\Gamma(B')$ with length $r-1$. Furthermore, for $A'=\mu_{2}^{\epsilon}(A)$ (where $\epsilon=sgn(\cc_2)$), we have $A'_{i,j}={\cc'}_i^TA_0{\cc'}_j$ for $i,j \in P'$ (Lemma~\ref{lem:lem-A'c'}) and $A'_{1,3},A'_{{r-1},r}>0$ (by definition). Then the statement follows by induction on $r$.


Let us now assume that $C: 1\to 2\to...\to r\to 1$, $r\geq 3$ is an oriented cycle in $\Gamma$ such that $A_{1,2},A_{{j-1},j}>0$, $j \geq 3$. 
If $j=3$ or $r=3$, then the first part of Proposition~\ref{prop:prop1} is contradicted; otherwise we could remove one of the edges in $C$ whose corresponding entry in $A$ is negative, then we obtain a path $P$ as in the first part of the proof. This completes the proof of the lemma. 
\end{proof}

\begin{lemma} 
\label{lem:lem-cycle} 
In the set-up of Theorem~\ref{th:companion0}, suppose that $A=(\cc_i^TA_0\cc_j)$ is a quasi-Cartan companion of $B$ and let $C$ be a cycle in $\Gamma(B)$. Then we have the following: 

If $C$ is non-oriented, then it has an even number of edges $\{i,j\}$ such that $A_{i,j}>0$; if $C$ is oriented, then it contains exactly one edge $\{i,j\}$ such that $A_{i,j}>0$. In particular, $A$ is an admissible quasi-Cartan companion of $B$.
\end{lemma} 

\begin{proof} 
Let us first suppose that $C=\{c_1,...,c_r\}$, $r\geq 3$, is non-oriented and for exactly an odd number $m$ of edges $\{i,j\}$ in $C$ we have $A_{i,j}>0$. 
We will arrive at a contradiction. Let us first assume that $r=3$. We may also assume without loss of generality that $k=c_1$ is the vertex which is neither source nor sink in $C$. Let
$\mu_k(\cc,B)=(\cc',B')$ with $B'=\mu_k(B)$ and let $A'=\mu^\epsilon_{k}(A)$. Then $A'_{i,j}={\cc'}_i^TA_0{\cc'}_j$ for $i,j \in C$ by Lemma~\ref{lem:lem-A'c'}. Furthermore,
$|B'_{c_2,c_3}|=|B_{c_2,c_3}|+|B_{c_2,c_1}||B_{c_1,c_3}| \ne 0$ but $|A'_{c_2,c_3}|=||B_{c_2,c_3}|-|B_{c_2,c_1}||B_{c_1,c_3}||$ (by definition of mutations); so $0\ne |B'_{c_2,c_3}|\ne |A'_{c_2,c_3}|$. This contradicts Proposition~\ref{prop:prop1}.

Let us now assume that $r>3$. Note that if $k$ is a source or sink then for, $\mu_k(\cc,B)=(\cc',B')$ and $A'=\mu^\epsilon_{k}(A)$,
the diagram $C'=\mu_k(C)$ is also a non-oriented cycle with an odd number of edges $\{i,j\}$ such that $A'_{i,j}>0$ (with $A'_{i,j}=({\cc'}_i^TA_0{\cc'}_j)$ for $i,j \in C$ by Lemma~\ref{lem:lem-A'c'}). Thus, to arrive at a contradiction, we may assume without loss of generality that 
$k$ is neither a source nor a sink in $C$. Note then that, for one of the edges, say $\{i,k\}$, incident to $k$ in $C$, we have $A_{i,k}<0$  (Proposition~\ref{prop:prop1}). 
Again let $\mu_k(\cc,B)=(\cc',B')$ and $A'=\mu^\epsilon_{k}(A)$ with $\epsilon=sgn(\cc_k)$. Then $C'=C-\{k\}$ is a non-oriented cycle in $\Gamma(B')$ with the same odd number of edges $\{i,j\}$ in $C'$ such that $A'_{i,j}>0$ (with $A'_{i,j}={\cc'}_i^TA_0{\cc'}_j$ for all $i,j$ in $C'$ by Lemma~\ref{lem:lem-A'c'}). Then the lemma follows by induction on $r$.



To show the statement for oriented cycles, we will show the following: 


(*) if $B$ and $A$ (as in the statement of the lemma) satisfy the sign condition of the lemma on oriented cycles of $\Gamma(B)$, then $B'=\mu_k(B)$ and $A'=\mu^\epsilon_k(A)$ satisfies the sign condition on oriented cycles of $\Gamma(B')$ (where $\epsilon=sgn(\cc_k)$).

For this suppose to the contrary that $B,A$ satisfy the sign condition but $B',A'$ do not. 
We will arrive at a contradiction, establishing (*). First let us note that, by the first part of the lemma, $A$ also satisfies the sign condition on non-oriented cycles. Then $A$ is, in particular, admissible. Thus $A'$ is a quasi-Cartan companion of $B'$ such that $A'_{i,j}={\cc'}_i^TA_0{\cc'}_j$ for all $i,j$ (by Proposition~\ref{prop:adm-triangle} and 
Lemma~\ref{lem:lem-A'c'}, here $\mu_k(\cc,B)=(\cc',B')$). Then, by the first part of the lemma, $A'$ satisfies the sign condition on non-oriented cycles.
Let us now suppose (to establish (*) by contradiction) that $A'$ does not satisfy the sign condition of the lemma on an oriented cycle $C'$ in $\Gamma(B')$. Then for any edge $\{i,j\}$ in $C'$ we have $A'_{i,j}<0$ by  Lemma~\ref{lem:lem-cycle0}. 
{Below we use the following properties of $B'$ and $A'$ without explicit reference}: $\mu_k(B')=B$ and $A=\mu_k^{\epsilon'}(A')$, where $\epsilon'=sgn({\cc'}_k)=-\epsilon=-sgn(\cc_k)$. 

Suppose first that $k$ is in $C'$. Then we have the following: If $C'=\{i,j,k\}$ is a triangle, then $|A_{i,j}|\ne |B_{i,j}|$, implying that $A$ is not a quasi-Cartan companion of $B$ (Lemma~\ref{prop:adm-triangle}), which contradicts our assumption. If $C'$ is not a triangle, then $C''=C'-\{k\}$ is an an oriented cycle in $\Gamma(B)$ such that for any edge $\{i,j\}$ in $C''$, we have $A_{i,j}<0$, contradicting our assumption (that $A$ satisfies the sign condition on oriented cycles). Thus, for the rest of the proof, we may assume that $k$ is not in $C'$, and, for any oriented cycle $C'''$ in $C'k:=C'\cup\{k\}$ that contains $k$, there is exactly one edge $\{i,j\}$ such that $A'_{i,j}>0$. We may also assume that $k$ is connected to at least two vertices in $C'$  with two opposite orientations (otherwise $k$ is a source or sink in $C'k$ and then $C'$ is also an oriented cycle in $\Gamma(B)$ with $A'_{i,j}=A_{i,j}$ for any edge $\{i,j\}$ in $C'$, contradicting our assumption that $A$ satisfies the sign condition of the lemma on oriented cycles. For convenience, we proceed arguing in cases.

\credit{Case 1} \emph{$k$ is connected to exactly two vertices in $C'$ and they are connected to each other.} Suppose that $k$ is connected to $\{l,l+1\}$ in $C'$. Then, by our assumption above, $k$ is connected to $\{l,l+1\}$ with opposite orientations. 

\smallskip
\noindent
{Subcase 1.} The triangle $T=\{l,l+1,k\}$ is non-oriented. Since $A'$ satisfies the sign condition of the lemma on $T$  and $A'_{l,l+1}<0$ by assumption, we have $sgn(A'_{l,k})=sgn(A'_{k,l+1})$; on the other hand, by Propositon~\ref{prop:prop1}, $A'_{l,k}<0$ or $A'_{k,l+1}<0$ (recall that $k$ is neither source nor sink in $C'k$). Thus $A'_{l,k},A'_{k,l+1}<0$. 
Then, by the definition of mutations, $C'$ is an oriented cycle in $\Gamma(B)$ such that for any edge $\{i,j\}$ in $C'$, we have $A_{i,j}<0$, contradicting our assumption.  

\smallskip
\noindent
{Subcase 2.} The triangle $T=\{l,l+1,k\}$ is oriented. Then, since $A'$ satisfies the sign condition of the lemma on $T$, we have $A'_{l,k}A'_{k,l+1}<0$.
Let us first assume that $|B'_{l,k}B'_{k,l+1}|<|B'_{l,l+1}|$. Then, by the definition of mutations, $C'$ is an oriented cycle in $\Gamma(B)$ such that for any edge $\{i,j\}$ in $C'$, we have $A_{i,j}<0$ (because $A_{l,l+1}<0$), contradicting our assumption. Similarly, if $|B'_{l,k}B'_{k,l+1}|=|B'_{l,l+1}|$, then $C'k$ is an oriented cycle in $\Gamma(B)$ such that for any edge $\{i,j\}$ in $C'k$, we have $A_{i,j}<0$, contradicting our assumption.  
Let us now assume that $|B'_{l,k}B'_{k,l+1}|>|B'_{l,l+1}|$. Then $C'$ is a non-oriented cycle in $\Gamma(B)$ such that $A_{l,l+1}>0$ but, for any edge $\{i,j\}\ne \{l,l+1\}$ in $C'$, we have $A_{i,j}<0$. This contradicts that $A$ satisfies the sign condition of the lemma on the non-oriented cycles.

\credit{Case 2} \emph{$k$ is connected to two vertices in $C'$ which are not connected to each other.} Let $c_1,c_2,...,c_r$, $r\geq 2$ be the vertices connected to $k$ ($c_i$'s may or may not be adjacent). Let $C_1,...,C_r$ be the cycles in $C'k$ that contain $k$ such that $C_i$ contains $\{k,c_i\}$ and $\{k,c_{i+1}\}$ with $c_{r+1}=c_1$.

\smallskip
\noindent
Subcase 1. One of $C_1,...,C_r$, say $C_1$, is oriented. 
Then $C_2$ and $C_r$ are non-oriented. Assume without loss of generality that $A'_{k,c_2}>0$ (then $A'_{k,c_1}<0$; recall that $\{k,c_2\}$ is in $C_2$). 
Then, since $A'$ satisfies the sign conditon of the lemma on the non-oriented cycle $C_2$, we have $A'_{k,c_3}>0$ and $\{k,c_3\}$ has the same orientation as $\{k,c_2\}$ (by Propositon~\ref{prop:prop1}). Then, similarly, $C_3$ is also non-oriented,
$A'_{k,c_4}>0$ and $\{k,c_4\}$ has the same orientation as $\{k,c_3\}$ and $\{k,c_2\}$. Continuing similarly, the cycle $C_{r-1}$ is also non-oriented and $A'_{k,c_r}>0$ and $\{k,c_r\}$ has the same orientation as $\{k,c_{r-1}\}$,..., $\{k,c_3\}$ and $\{k,c_2\}$.
Then $C_r$ is non-oriented with $A'_{i,j}<0$ for any $\{i,j\} \ne \{k,c_r\}$ in $C_r$ and $A'_{k,c_r}>0$; this contradicts our assumption that $A'$ satisfies the sign conditon of the lemma on the non-oriented cycles.

\smallskip
\noindent
Subcase 2. $C_1,...,C_r$ are non-oriented. Suppose first that there is $C_i$, say $C_1$, where $k$ is neither source nor sink, so $\{k,c_1\}$ and 
$\{k,c_2\}$ have opposite orientations. Then $\{k,c_3\}$ has the same orientation as $\{k,c_2\}$ (otherwise $C_2$ is oriented). Arguing similarly, the edge $\{k,c_r\}$ has the same orientation as $\{k,c_{r-1}\},...,\{k,c_2\}$. But then $C_r$ becomes oriented, contradicting this case.
If $k$ is either a source or sink in each $C_i$, then $k$ is a source or sink in $C'k$, contradicting our assumption.
This completes the case and the proof of the lemma.
\end{proof}


Let us now prove Theorem~\ref{th:companion0}. For this, we only need to prove that $A=({\cc_{i}}^TA_0\cc_j)$ is a quasi-Cartan companion of $B$; 
the other claims then follow from Proposition~\ref{prop:prop1}. 
To show that $A$ is a quasi-Cartan companion of $B$, we use an inductive argument. The basis of the induction is the fact that the generalized Cartan matrix $A_0$ is a quasi-Cartan companion of $B_0$. For the inductive step, 
suppose that $A=(\cc_i^TA_0\cc_j)$ is a quasi-Cartan companion of $B$. 
Then $A$ is, in particular, admissible by Lemma~\ref{lem:lem-cycle}; so $A'=\mu_k^\epsilon(A)$ is a quasi-Cartan companion of $B'=\mu_k(B)$ (where $\epsilon=sgn(\cc_k)$). This completes the proof of Theorem~\ref{th:companion0}. Then Corollary~\ref{cor:companion} follows from Lemma~\ref{lem:lem-A'c'}. Theorem~\ref{th:companion2} follows from Lemma~\ref{lem:lem-cycle}; here the fact that any two admissible quasi-Cartan companions can be obtained from each other by a sequence of simultaneous sign changes in rows and columns has been obtained in \cite[Theorem~2.11]{S3}. 
Also it can be easily checked that simultaneous sign changes in rows and columns commute with the mutation of quasi-Cartan companions. Then Corollary~\ref{cor:admissible2} follows. 






\end{document}